\documentclass[12pt]{article}

\usepackage{amsmath, epsfig, cite, setspace}
\usepackage{amssymb}
\usepackage{amsfonts}
\usepackage{latexsym}
\usepackage{amsthm}

\makeatletter
\renewcommand{\@seccntformat}[1]{{\csname the#1\endcsname}.\hspace{.5em}}
\makeatother

\newtheorem{thm}{Theorem}[section]

\newtheorem{lem}[thm]{Lemma}
\newtheorem{remark}[thm]{Remark}

\renewcommand{\qed}{\hfill$\Box$\medskip}
\renewcommand{\thefootnote}{*}

\numberwithin{equation}{section}

\begin{document}

\begin{center}
{\large\bf Proof of two supercongruences conjectured by Z.-W. Sun}
\end{center}

\vskip 2mm \centerline{Guo-Shuai Mao}
\begin{center}
{\footnotesize $^1$Department of Mathematics, Nanjing
University of Information Science and Technology, Nanjing 210044,  People's Republic of China\\
{\tt maogsmath@163.com  } }
\end{center}
\vskip 2mm \centerline{Chen-Wei Wen}
\begin{center}
{\footnotesize $^2$Department of Mathematics, Nanjing
University, Nanjing 210093,  People's Republic of China\\
{\tt wenchenwei@126.com } }
\end{center}

\vskip 0.7cm \noindent{\bf Abstract.}
In this paper, we prove two supercongruences conjectured by Z.-W. Sun via the Wilf-Zeilberger method. One of them is, for any prime $p>3$,
\begin{align*}
\sum_{n=0}^{p-1}\frac{6n+1}{256^n}\binom{2n}n^3&\equiv p(-1)^{(p-1)/2}-p^3E_{p-3}\pmod{p^4}.
\end{align*}
In fact, this supercongruence is a generalization of a supercongruence of van Hamme.

\vskip 3mm \noindent {\it Keywords}: Supercongruence; Binomial coeficients; Wilf-Zeilberger method; Euler numbers.

\vskip 0.2cm \noindent{\it AMS Subject Classifications:} 11B65, 11A07, 11B68, 33F10, 05A10.

\renewcommand{\thefootnote}{**}

\section{Introduction}

\qquad In the past decade, many researchers studied supercongruences via the Wilf-Zeilberger (WZ) method. For instance, W. Zudilin \cite{zudilin-jnt-2009} proved several Ramanujan-type supercongruences by the WZ method. One of them, conjectured by van Hamme, says that for any odd prime $p$,
\begin{align}\label{wzpr}
\sum_{k=0}^{(p-1)/2}(4k+1)(-1)^k\left(\frac{\left(\frac12\right)_k}{k!}\right)^3\equiv(-1)^{(p-1)/2}p\pmod{p^3},
\end{align}
where $(a)_n=a(a+1)\ldots(a+n-1) (n\in\{1,2,\ldots\})$ with $(a)_0=1$ is the raising factorial for $a\in\mathbb{C}$.

Chen, Xie and He \cite{CXH-rama-2016} confirmed a supercongruence conjetured by Z.-W. Sun \cite{sun-scm-2011}, which says that for any prime $p>3$,
$$
\sum_{k=0}^{p-1}\frac{3k+1}{(-8)^k}\binom{2k}k^3\equiv p(-1)^{(p-1)/2}+p^3E_{p-3}\pmod{p^4},
$$
where $E_n$ are the Euler numbers defined by
$$E_0=1,\ \mbox{and}\ E_n=-\sum_{k=1}^{\lfloor n/2\rfloor}\binom{n}{2k}E_{n-2k}\ \mbox{for}\ n\in\{1,2,\ldots\}.$$

For $n\in\mathbb{N}$, define
$$H_n:=\sum_{0<k\leq n}\frac1k,\ H_n^{(2)}:=\sum_{0<k\leq n}\frac1{k^2},\ H_0=H_0^{(2)}=0,$$
where $H_n$ with $n\in\mathbb{N}$ are often called the classical harmonic numbers. Let $p>3$ be a prime. J. Wolstenholme \cite{wolstenholme-qjpam-1862} proved that
$$H_{p-1}\equiv0\pmod{p^2}\ \mbox{and}\ H_{p-1}^{(2)}\equiv0\pmod p,$$
which imply that
\begin{align}
\binom{2p-1}{p-1}\equiv1\pmod{p^3}.\label{2p1p}
\end{align}

Z.-W. Sun \cite{sun-ijm-2012} proved the following supercongruence by the WZ method, for any odd prime $p$,
\begin{equation}\label{sun}
\sum_{k=0}^{p-1}\frac{4k+1}{(-64)^k}\binom{2k}k^3\equiv(-1)^{\frac{(p-1)}2}p+p^3E_{p-3}\pmod{p^4}.
\end{equation}

Guo and Liu \cite{gl-arxiv-2019} showed that for any prime $p>3$,
\begin{equation}\label{glp4}
\sum_{k=0}^{(p+1)/2}(-1)^k(4k-1)\frac{\left(-\frac12\right)_k^3}{(1)_k^3}\equiv p(-1)^{(p+1)/2}+p^3(2-E_{p-3})\pmod{p^4}.
\end{equation}
Guo also researched $q$-analogues of Ramanujan-type supercongruences and $q$-analogues of supercongruences of van Hamme (see, for instance, \cite{g-jmaa-2018,guo-rama-2019}).

Long \cite{long-2011-pjm} proved a conjecture of van Hamme \cite{vhamme}, which is, for any prime $p>3$,
$$
\sum_{n=0}^{(p-1)/2}\frac{6n+1}{256^n}\binom{2n}n^3\equiv p(-1)^{(p-1)/2}\pmod{p^4}.
$$

In this paper, we first obtain the following result which confirms a conjecture of Sun\cite{sun-scm-2011}:
\begin{thm}\label{Thsun} Let $p>3$ be a prime. Then
\begin{equation}\label{6n1256}
\sum_{n=0}^{p-1}\frac{6n+1}{256^n}\binom{2n}n^3\equiv p(-1)^{(p-1)/2}-p^3E_{p-3}\pmod{p^4}.
\end{equation}
\end{thm}
 Via the WZ method, Zudilin \cite{zudilin-jnt-2009} proved that for any odd prime $p$,
\begin{align}\label{zu}
\sum_{n=0}^{p-1}\frac{\left(\frac12\right)_n\left(\frac12\right)_{2n}}{n!^3}\frac{20n+3}{2^{4n}}\equiv3 p(-1)^{(p-1)/2}\pmod{p^3}.
\end{align}
Z.-W. Sun \cite{sun-ijm-2012} used the WZ pair which was used by Zudilin \cite{zudilin-jnt-2009} to prove the following supercongruence, for any prime $p>3$,
\begin{equation*}
\sum_{k=0}^{(p-1)/2}\frac{20k+3}{(-2^{10})^k}{4k\choose k,k,k,k}\equiv p(-1)^{\frac{(p-1)}2}(2^{p-1}+2-(2^{p-1}-1)^2)\pmod{p^4}.
\end{equation*}

Our second result is the following congruence which generalizes Sun's result.
\begin{thm}\label{Thsun2} Let $p$ be an odd prime. Then
\begin{equation*}
\sum_{n=0}^{p-1}\frac{20n+3}{(-2^{10})^n}{4n\choose n,n,n,n}\equiv 3p(-1)^{(p-1)/2}+3p^3E_{p-3}\pmod{p^4}.
\end{equation*}
\end{thm}
\begin{remark}\rm In \cite{sun-ijm-2012}, Sun said that he ever proved the congruence in Theorem \ref{Thsun2}, but he lost the draft containing the complicated details. Recently, he told us to prove this congruence.
\end{remark}
Our main tool in this paper is the WZ method. We shall prove Theorem \ref{Thsun} in the next Section, and the last Section is devoted to prove Theorem \ref{Thsun2}.
\section{Proof of Theorem \ref{Thsun}}
\qquad We will use the following WZ pair which appears in \cite{he-jnt-2015} to prove Theorem \ref{Thsun}. For nonnegative integers $n, k$, define
$$
F(n,k)=\frac{(6n-2k+1)}{2^{8n-2k}}\frac{\binom{2n}n\binom{2n+2k}{n+k}\binom{2n-2k}{n-k}\binom{n+k}{n}}{\binom{2k}k}
$$
and
$$
G(n,k)=\frac{n^2\binom{2n}n\binom{2n+2k}{n+k}\binom{2n-2k}{n-k}\binom{n+k}{n}}{2^{8n-2k-4}(2n+2k-1)\binom{2k}k}.
$$
Clearly $F(n,k)=G(n,k)=0$ if $n<k$. It is easy to check that
\begin{equation}\label{FG}
F(n,k-1)-F(n,k)=G(n+1,k)-G(n,k)
\end{equation}
for all nonnegative integer $n$ and $k>0$.

Summing (\ref{FG}) over $n$ from $0$ to $p-1$ we have
$$
\sum_{n=0}^{p-1}F(n,k-1)-\sum_{n=0}^{p-1}F(n,k)=G(p,k)-G(0,k)=G(p,k).
$$
Furthermore, summing both side of the above identity over $k$ from $1$ to $p-1$, we obtain
\begin{align}\label{wz1}
\sum_{n=0}^{p-1}F(n,0)=F(p-1,p-1)+\sum_{k=1}^{p-1}G(p,k).
\end{align}
\begin{lem} \label{Fp1p1} Let $p>3$ be a prime. Then
$$F(p-1,p-1)\equiv -3p^2-12p^3+18p^3q_p(2)\pmod{p^4},$$
where $q_p(2)$ stands for the Fermat quotient $(2^{p-1}-1)/p$.
\end{lem}
\begin{proof} By the definition of $F(n,k)$, we have
\begin{align*}
F(p-1,p-1)&=\frac{4p-3}{2^{6p-6}}\binom{4p-4}{2p-2}\binom{2p-2}{p-1}=\frac{p\binom{2p-1}{p-1}\binom{4p-3}{2p-2}}{2^{6p-6}}=\frac{p^2\binom{2p-1}{p-1}\binom{4p-1}{2p-1}}{(4p-1)2^{6p-6}}.
\end{align*}
It is known that Jocobsthal's congruence is as follows: For primes $p>3$, integers $a,b$ and integers $r,s\geq1$,
\begin{align}\label{Jacob}
\binom{ap^r}{bp^s}/\binom{ap^{r-1}}{bp^{s-1}}\equiv1 \pmod{p^{r+s+\tt{min}\{r,s\}}}.
\end{align}
Thus,
$$\binom{4p-1}{2p-1}=\frac12\binom{4p}{2p}\equiv\frac12\binom{4}{2}\equiv3\pmod{p^3}.$$
This, with (\ref{2p1p}) and $2^{p-1}=1+pq_p(2)$ yields that
$$
F(p-1,p-1)\equiv\frac{3p^2}{(4p-1)2^{6p-6}}\equiv-3p^2-12p^3+18p^3q_p(2)\pmod{p^4}.
$$
Therefore the proof of Lemma \ref{Fp1p1} is complete.
\end{proof}
By the definition of $G(n,k)$ we have
\begin{align}\label{Gpk}
G(p,k)&=\frac{p^2\binom{2p}{p}\binom{2p+2k}{p+k}\binom{2p-2k}{p-k}\binom{p+k}{p}}{2^{8p-4-2k}(2p+2k-1)\binom{2k}k}=\frac{p^2\binom{2p}{p}\binom{2p+2k}{p}\binom{2p-2k}{p-k}\binom{p+2k}{k}}{2^{8p-4-2k}(2p+2k-1)\binom{2k}k},
\end{align}
where we used the binomial transformation
$$
\binom{n}{k}\binom{k}{j}=\binom{n}{j}\binom{n-j}{k-j}.
$$
\begin{lem}\label{sunimp}{\rm (\cite[(3.1)]{sun-jnt-2011})}
\begin{equation*}
\sum_{k=1}^{(p-1)/2}\frac{4^k}{k(2k-1)\binom{2k}k}\equiv2E_{p-3}\pmod p.
\end{equation*}
\end{lem}
\begin{remark}\rm The congruence in Lemma \ref{sunimp} is often used when we use the WZ method to prove some supercongruences. For instance, see \cite{hm-rama-2017,CXH-rama-2016}.

\end{remark}
\begin{lem} \label{G12} For any primes $p>3$, we have
$$
\sum_{k=1}^{(p-1)/2}G(p,k)\equiv -p^3E_{p-3}\pmod{p^4}.
$$
\end{lem}
\begin{proof}
It is easy to see that $\binom{2p-2k}{p-k}\equiv0\pmod p$ for each $1\leq k\leq(p-1)/2$.
Then by (\ref{Gpk}), (\ref{Jacob}) and Lucas congruence, we have
\begin{align*}
G(p,k)\equiv\frac{4p^{2}\binom{2p-2k}{p-k}}{(2p+2k-1)2^{8p-4-2k}}\pmod{p^4}.
\end{align*}
In view of \cite[Lemma 2.1]{sun-scm-2011},
$$k\binom{2k}k\binom{2(p-k)}{p-k}\equiv(-1)^{\lfloor2k/p\rfloor-1}2p\pmod{p^2}\ \text{for all}\ k=1,\ldots,p-1.$$
Then for each $1\leq k\leq(p-1)/2$ we have
\begin{align*}
\binom{2p-2k}{p-k}\equiv\frac{-2p}{k\binom{2k}k}\pmod {p^2}.
\end{align*}
Hence
$$
G(p,k)\equiv\frac{-p^3}{2^{8p-7-2k}k(2k-1)\binom{2k}k}\equiv-\frac{p^3}2\frac{4^k}{k(2k-1)\binom{2k}k} \pmod{p^4}.
$$
Therefore we immediately obtain the desired result with Lemma \ref{sunimp}.
\end{proof}
\begin{lem}\label{Gpr} Let $p>3$ be a prime. Then
$$
G(p,(p+1)/2)\equiv(-1)^{(p-1)/2}p\left(1-3pq_p(2)+6p^2q_p(2)^2\right)\pmod{p^4}.
$$
\end{lem}
\begin{proof}
In view of (\ref{Gpk}) and (\ref{Jacob}), we have
\begin{align*}
G(p,(p+1)/2)&=\frac{p^2\binom{2p}{p}\binom{3p+1}{p}\binom{p-1}{(p-1)/2}\binom{2p+1}{(p+1)/2}}{2^{7p-5}(2p+p)\binom{p+1}{(p+1)/2}}=\frac{p^2\binom{2p}{p}\binom{3p+1}{2p+1}\binom{p-1}{(p-1)/2}\binom{2p+1}{(p+1)/2}}{2^{7p-5}(2p+p)\binom{p+1}{(p+1)/2}}\\
&=\frac{(3p+1)(p+1)\binom{2p+1}{(p+1)/2}}{(2p+1)2^{7p-4}}=\frac{(3p+1)\binom{2p}{(3p+1)/2}}{2^{7p-5}}=\frac{p\binom{2p-1}{(3p-1)/2}}{2^{7p-7}}.
\end{align*}
It is easy to see that
\begin{align*}
\binom{2p-1}{(p-1)/2}\equiv(-1)^{(p-1)/2}\left(1-2pH_{(p-1)/2}+2p^2H_{(p-1)/2}^2-2p^2H_{(p-1)/2}^{(2)}\right)\pmod{p^3}.
\end{align*}
In view of \cite{zhsun}, we have
$$
H_{(p-1)/2}\equiv-2q_p(2)+pq_p(2)^2\pmod{p^2}\ \mbox{and}\ H_{(p-1)/2}^{(2)}\equiv0\pmod p.
$$
Thus,
$$\binom{2p-1}{(p-1)/2}\equiv(-1)^{(p-1)/2}\left(1+4pq_p(2)+6p^2q_p(2)^2\right)\pmod{p^3}.$$
Therefore
\begin{align*}
G(p,(p+1)/2)\equiv(-1)^{(p-1)/2}p\left(1-3pq_p(2)+6p^2q_p(2)^2\right)\pmod{p^4}
\end{align*}
since $1/2^{7p-7}\equiv1-7pq_p(2)+28p^2q_p(2)^2\pmod{p^3}$.
\end{proof}
\begin{lem}\label{sun1}{\rm (\cite[(1.1) and (1.7)]{sun-jnt-2011})} Let $p>3$ be a prime. Then
$$
\sum_{k=0}^{(p-3)/2}\frac{\binom{2k}k}{(2k+1)4^k}\equiv-(-1)^{(p-1)/2}q_p(2)\pmod{p^2},
$$
$$
\sum_{k=0}^{(p-3)/2}\frac{\binom{2k}k}{(2k+1)^24^k}\equiv-(-1)^{(p-1)/2}\frac{q_p(2)^2}2\pmod{p}.
$$
\end{lem}
\begin{lem}\label{humao} {\rm (\cite{hm-rama-2017})}
$$
\sum_{k=1}^{n}\frac{(-1)^k\binom{n}k H_k}{(2k+1)}=-\frac{4^n}{(2n+1)\binom{2n}n}\sum_{k=1}^n\frac{\binom{2k}k}{k4^k},
$$
$$
\sum_{k=1}^{n}\frac{(-1)^k\binom{n}k H_{2k}}{(2k+1)}=-\frac{4^n}{(2n+1)\binom{2n}n}\left(\frac{H_n}2+\frac12\sum_{k=1}^n\frac{\binom{2k}k}{k4^k}\right).
$$
\end{lem}
\begin{lem}\label{mor}{\rm (\cite{Mor})} For any prime $p>3$, we have
$$
\binom{p-1}{(p-1)/2}\equiv(-1)^{(p-1)/2}4^{p-1}\pmod{p^3}.
$$
\end{lem}
\begin{lem}\label{Gpr12pr} For any prime $p>3$, we have
$$\sum_{k=(p+3)/2}^{p-1}G(p,k)\equiv3p^2(1+4p-6pq_p(2))+(-1)^{(p-1)/2}3p^{2}q_p(2)(1-2pq_p(2))\pmod{p^{4}}.$$
\end{lem}
\begin{proof}
Again by (\ref{Gpk}), we have
$$
\sum_{k=\frac{p+3}2}^{p-1}G(p,k)=\frac{2p^2}{2^{8p-4}}\sum_{k=\frac{p+3}2}^{p-1}\frac{4^k\binom{2p+2k}{p}\binom{p+2k}k\binom{2p-2k}{p-k}}{(2p+2k-1)\binom{2k}k}=\frac{2p^2}{2^{6p-4}}\sum_{k=1}^{\frac{p-3}2}\frac{\binom{4p-2k}{p}\binom{3p-2k}{p-k}\binom{2k}k}{4^k(4p-2k-1)\binom{2p-2k}{p-k}}.
$$
It is easy to check that
\begin{align*}
&\binom{4p-2k}{p}\binom{3p-2k}{p-k}=\frac{(4p-2k)\ldots(2p-k+1)}{p!(p-k)!}\\
&=6p\frac{(3p+p-2k)\ldots(3p+1)(3p-1)\ldots(2p+1)(2p-1)\ldots(2p-(k-1))}{(p-1)!(p-k)!}\\
&\equiv6p\frac{(p-2k)!(1+3pH_{p-2k})(p-1)!(1+2pH_{p-1})(-1)^k(k-1)!(1-2pH_{k-1})}{(p-1)!(p-k)!}\\
&\equiv\frac{(-1)^k6p(1+3pH_{p-2k}-2pH_{k-1})}{k\binom{p-k}k}\pmod{p^3}.
\end{align*}
Hence
\begin{align*}
\sum_{k=(p+3)/2}^{p-1}G(p,k)\equiv\frac{-12p^3}{2^{6p-4}}\sum_{k=1}^{(p-3)/2}\frac{(1+3pH_{p-2k}-2pH_{k-1})\binom{2k}k}{(-4)^k(4p-2k-1)k\binom{2p-2k}{p-k}\binom{p-k}k}\pmod{p^4}.
\end{align*}
By simple computation, we have
\begin{align*}
\frac{\binom{2p-2k}{p-k}\binom{p-k}k}{p}&=\frac{(2p-2k)\ldots(p+1)(p-1)\ldots(p-2k+1)}{(p-k)!k!}\\
&\equiv-\frac{(1+pH_{p-2k}-pH_{2k-1})\binom{2k}k}{(2k)\binom{p-k}k}\pmod {p^2}
\end{align*}
and
\begin{align*}
\binom{p-k}{k}=\frac{(p-k)\ldots(p-2k+1)}{k!}&=\frac{(-1)^k(2k-1)!(1-p(H_{2k-1}-H_{k-1}))}{k!(k-1)!}\\
&\equiv\frac{(-1)^k}2\binom{2k}k(1-pH_{2k-1}+pH_{k-1})\pmod {p^2}.
\end{align*}
Thus
\begin{align*}
\sum_{k=(p+3)/2}^{p-1}G(p,k)&\equiv\frac{3p^2}{2^{6p-6}}\sum_{k=1}^{(p-3)/2}\frac{(1+3pH_{p-2k}-2pH_{k-1}-pH_{2k-1}+pH_{k-1})\binom{2k}k}{4^k(4p-2k-1)}\\
&\equiv\frac{3p^2}{2^{6p-6}}\sum_{k=1}^{(p-3)/2}\frac{(1+2pH_{2k}-pH_{k})\binom{2k}k}{4^k(4p-2k-1)}\pmod{p^4},
\end{align*}
where we used $H_{p-1-k}\equiv H_k\pmod p$ for all $k\in\{0,1,\ldots,p-1\}$.

Continuing to calculate the above congruence, we have the following congruence modulo $p^4$
\begin{align*}
\sum_{k=\frac{p+3}2}^{p-1}G(p,k)&\equiv\frac{-3p^2}{2^{6p-6}}\left(\sum_{k=1}^{\frac{p-3}2}\frac{\binom{2k}k}{(2k+1)4^k}+4p\sum_{k=1}^{\frac{p-3}2}\frac{\binom{2k}k}{(2k+1)^24^k}+p\sum_{k=1}^{\frac{p-3}2}\frac{\binom{2k}k(2H_{2k}-H_k)}{(2k+1)4^k}\right).
\end{align*}
In view of Lemma \ref{humao}, we have
\begin{align*}
\sum_{k=1}^{\frac{p-3}2}\frac{(-1)^k\binom{\frac{p-1}2}k(2H_{2k}-H_k)}{(2k+1)}&\equiv\sum_{k=1}^{\frac{p-1}2}\frac{(-1)^k\binom{\frac{p-1}2}k(2H_{2k}-H_k)}{(2k+1)}-\frac{(-1)^{\frac{p-1}2}}{p}(2H_{p-1}-H_{\frac{p-1}2})\\
&\equiv-\frac{2^{p-1}}{p\binom{p-1}{(p-1)/2}}H_{(p-1)/2}+\frac{(-1)^{\frac{p-1}2}}{p}H_{\frac{p-1}2}\pmod{p}.
\end{align*}
Hence by Lemma \ref{mor}, $2^{p-1}=1+pq_p(2)$ and $H_{(p-1)/2}\equiv-2q_p(2)\pmod p$
\begin{align}\label{H2kHk}
\sum_{k=1}^{\frac{p-3}2}\frac{(-1)^k\binom{\frac{p-1}2}k(2H_{2k}-H_k)}{(2k+1)}\equiv-2(-1)^{(p-1)/2}q_p(2)^2\pmod p.
\end{align}
Therefore we immediately get the desired result with Lemma \ref{sun1}.
\end{proof}
\noindent{\it Proof of Theorem \ref{Thsun}}. Combining (\ref{wz1}) with Lemmas \ref{Fp1p1}, \ref{G12}, \ref{Gpr} and \ref{Gpr12pr}, we immediately obtain that
$$
\sum_{k=0}^{p-1}\frac{6k+1}{256^k}\binom{2k}k^3\equiv(-1)^{(p-1)/2}p-p^3E_{p-3}\pmod{p^4}.
$$
Therefore the proof of Theorem \ref{Thsun} is finished. \qed
\section{Proof of Theorem \ref{Thsun2}}
\qquad To prove Theorem \ref{Thsun2}, we should use the following WZ pair which appears in \cite{sun-ijm-2012} (see also \cite{zudilin-jnt-2009}). For nonnegative integers $n, k$, define
$$
F(n,k)=(-1)^{n+k}\frac{(20n-2k+3)}{4^{5n-k}}\frac{\binom{2n}n\binom{4n+2k}{2n+k}\binom{2n-k}{n}\binom{2n+k}{2k}}{\binom{2k}k}
$$
and
$$
G(n,k)=(-1)^{n+k}\frac{n\binom{2n-1}{n-1}\binom{4n+2k-2}{2n+k-1}\binom{2n-k-1}{n-1}\binom{2n+k-1}{2k}}{4^{5n-k-4}\binom{2k}k}.
$$
Clearly $F(n,k)=G(n,k)=0$ if $n<k$. It is easy to check that
\begin{equation}\label{FG1}
F(n,k-1)-F(n,k)=G(n+1,k)-G(n,k)
\end{equation}
for all nonnegative integer $n$ and $k>0$.

Summing (\ref{FG1}) over $n$ from $0$ to $p-1$ and then over $k$ from $1$ to $p-1$, we have
\begin{align}\label{wz2}
\sum_{n=0}^{p-1}F(n,0)=F(p-1,p-1)+\sum_{k=1}^{p-1}G(p,k).
\end{align}
\begin{lem} \label{Fp1p12} Let $p>3$ be a prime. Then
$$F(p-1,p-1)\equiv 15p^2(-1-6p+8pq_p(2))\pmod{p^4}.$$
\end{lem}
\begin{proof} By the definition of $F(n,k)$, we have
\begin{align*}
F(p-1,p-1)&=\frac{18p-15}{4^{4p-4}}\binom{6p-6}{3p-3}\binom{3p-3}{p-1}=\frac{3p\binom{3p-2}{2p-2}\binom{6p-5}{3p-3}}{4^{4p-4}}=\frac{p\binom{3p-1}{2p-1}\binom{6p-3}{3p-2}}{2\cdot4^{4p-4}}\\
&=\frac{p\binom{3p-1}{2p-1}\binom{6p-2}{3p-1}}{4^{4p-3}}=\frac{3p^2\binom{3p-1}{2p-1}\binom{6p-1}{3p-1}}{(6p-1)4^{4p-3}}=\frac{p^2\binom{3p}{2p}\binom{6p}{3p}}{(6p-1)4^{4p-3}}.
\end{align*}
In view of (\ref{Jacob}), we have
$$\binom{3p}{2p}\binom{6p}{3p}\equiv\binom{3}{2}\binom{6}{3}\equiv60\pmod{p^3}.$$
This, with $2^{p-1}=1+pq_p(2)$ yields that
$$
F(p-1,p-1)\equiv\frac{15p^2}{(6p-1)4^{4p-4}}\equiv15p^2(-1-6p+8pq_p(2))\pmod{p^4}.
$$
Therefore the proof of Lemma \ref{Fp1p12} is complete.
\end{proof}

 By the definition of $G(n,k)$ we have
\begin{align*}
G(p,k)&=\frac{(-1)^{k+1}p\binom{2p-1}{p-1}\binom{4p+2k-2}{2p+k-1}\binom{2p-k-1}{p-1}\binom{2p+k-1}{2k}}{4^{5p-4-k}\binom{2k}k}\notag\\
&=\frac{(-1)^{k+1}p\binom{2p-1}{p-1}\binom{4p+2k-2}{2k}\binom{2p-k-1}{p-1}\binom{4p-2}{2p-k-1}}{4^{5p-4-k}\binom{2k}k}\notag\\
&=\frac{(-1)^{k+1}p\binom{2p-1}{p-1}\binom{4p+2k-2}{2k}\binom{4p-2}{p-1}\binom{3p-1}{p-k}}{4^{5p-4-k}\binom{2k}k},
\end{align*}
where we used the binomial transformation
$$
\binom{n}{k}\binom{k}{j}=\binom{n}{j}\binom{n-j}{k-j}.
$$
Note that $\binom{n}k=(-1)^k\binom{-n+k-1}k$, then we have
$$
\binom{4p+2k-2}{2k}=\binom{-4p+1}{2k}=\frac{4p(4p-1)}{2k(2k-1)}\binom{-4p-1}{2k-2}.
$$
So by (\ref{Jacob}) we have
\begin{align}\label{Gpk1}
G(p,k)\equiv\frac{(-1)^{k+1}6p^3\binom{-4p-1}{2k-2}\binom{3p-1}{p-k}}{4^{5p-4-k}k(2k-1)\binom{2k}k}\pmod{p^4}.
\end{align}
\begin{lem} \label{G13} For any primes $p>3$, we have
$$
\sum_{k=1}^{(p-1)/2}G(p,k)\equiv -p^3E_{p-3}\pmod{p^4}.
$$
\end{lem}
\begin{proof}
It is easy to see that
\begin{align*}
\binom{-4p-1}{2k-2}\binom{3p-1}{p-k}\equiv(-1)^{p-k}\pmod p.
\end{align*}
Then by (\ref{Gpk1}), we have
$$
G(p,k)\equiv\frac{3p^34^k}{2k(2k-1)\binom{2k}k}\pmod{p^4}.
$$
Therefore we immediately obtain the desired result with Lemma \ref{sunimp}.
\end{proof}
\begin{lem}\label{Gpr1} Let $p>3$ be a prime. Then
$$
G(p,(p+1)/2)\equiv(-1)^{(p-1)/2}3p\left(1-5pq_p(2)+15p^2q_p(2)^2\right)\pmod{p^4}.
$$
\end{lem}
\begin{proof}
In view of (\ref{Gpk1}), we have the following congruence modulo $p^4$
\begin{align*}
G(p,(p+1)/2)&\equiv\frac{(-1)^{(p-1)/2}6p^3\binom{-4p-1}{p-1}\binom{3p-1}{(p-1)/2}}{2^{9p-9}\frac{p+1}2p\binom{p+1}{(p+1)/2}}=\frac{(-1)^{(p-1)/2}3p\binom{-4p-1}{p-1}\binom{3p-1}{(p-1)/2}}{2^{9p-9}\binom{p-1}{(p-1)/2}}.
\end{align*}
It is easy to see that
\begin{align*}
\binom{3p-1}{(p-1)/2}\equiv(-1)^{(p-1)/2}\left(1-3pH_{(p-1)/2}+3p^2H_{(p-1)/2}^2-3p^2H_{(p-1)/2}^{(2)}\right)\pmod{p^3}.
\end{align*}
In view of \cite{zhsun}, we have
$$
H_{(p-1)/2}\equiv-2q_p(2)+pq_p(2)^2\pmod{p^2}\ \mbox{and}\ H_{(p-1)/2}^{(2)}\equiv0\pmod p.
$$
Thus,
$$\binom{3p-1}{(p-1)/2}\equiv(-1)^{(p-1)/2}\left(1+6pq_p(2)+15p^2q_p(2)^2\right)\pmod{p^3}.$$
Therefore by Lemma \ref{mor}, (\ref{Jacob}) and $1/2^{11p-11}\equiv1-11pq_p(2)+66p^2q_p(2)^2\pmod{p^3}$, we have
\begin{align*}
G(p,(p+1)/2)\equiv(-1)^{(p-1)/2}3p\left(1-5pq_p(2)+15p^2q_p(2)^2\right)\pmod{p^4}.
\end{align*}
Now the proof of Lemma \ref{Gpr1} is completed.
\end{proof}
\begin{lem}\label{Gpr12pr1} For any prime $p>3$, we have
$$\sum_{k=(p+3)/2}^{p-1}G(p,k)\equiv15p^2(1+6p-8pq_p(2))+(-1)^{(p-1)/2}15p^{2}(q_p(2)-3pq_p(2)^2)\pmod{p^{4}}.$$
\end{lem}
\begin{proof}
Again by (\ref{Gpk1}), we have
\begin{align*}
\sum_{k=\frac{p+3}2}^{p-1}G(p,k)&\equiv-\frac{6p^3}{4^{5p-4}}\sum_{k=\frac{p+3}2}^{p-1}\frac{(-4)^{k}\binom{-4p-1}{2k-2}\binom{3p-1}{p-k}}{k(2k-1)\binom{2k}k}\\
&=-\frac{6p^3}{4^{5p-4}}\sum_{k=1}^{(p-3)/2}\frac{(-4)^{p-k}\binom{-4p-1}{2p-2k-2}\binom{3p-1}{k}}{(p-k)(2p-2k-1)\binom{2p-2k}{p-k}}\\
&=\frac{6p^3}{4^{4p-4}}\sum_{k=1}^{(p-3)/2}\frac{\binom{-4p-1}{2p-2k-2}\binom{3p-1}{k}}{(-4)^k(p-k)(2p-2k-1)\binom{2p-2k}{p-k}}\pmod{p^4}.
\end{align*}
It is easy to check that
\begin{align*}
\binom{3p-1}k=\prod_{j=1}^k\frac{3p-j}j=(-1)^k\prod_{j=1}^k\left(1-\frac{3p}j\right)\equiv(-1)^k(1-3pH_{p-2k})\pmod{p^2}
\end{align*}
and
\begin{align*}
\binom{-4p-1}{2p-2k-2}&=\prod_{j=1}^{2p-2k-2}\frac{-4p-j}j=\prod_{j=1}^{2p-2k-2}\left(1+\frac{4p}j\right)\\
&=5\prod_{j=1}^{p-1}\left(1+\frac{4p}j\right)\prod_{j=p+1}^{2p-2k-2}\left(1+\frac{4p}j\right)\\
&\equiv5(1+4pH_{p-1})(1+4pH_{p-2k-2})\equiv5(1+4pH_{p-2k-2})\pmod{p^2}.
\end{align*}
Hence
\begin{align*}
\sum_{k=(p+3)/2}^{p-1}G(p,k)\equiv\frac{30p^3}{4^{4p-4}}\sum_{k=1}^{(p-3)/2}\frac{(1-3pH_{k}+4pH_{p-2k-2})}{4^k(p-k)(2p-2k-1)\binom{2p-2k}{p-k}}\pmod{p^4}.
\end{align*}
By simple computation, we have
\begin{align*}
\frac{\binom{2p-2k}{p-k}}p&=\frac{(2p-2k)\ldots(p+1)(p-1)\ldots(p-k+1)}{(p-k)!}\\
&\equiv(-1)^{k-1}\frac{(1+pH_{p-2k}-pH_{k-1})}{k\binom{p-k}k}\pmod {p^2}
\end{align*}
and
\begin{align*}
\binom{p-k}{k}=\frac{(p-k)\ldots(p-2k+1)}{k!}&=\frac{(-1)^k(2k-1)!(1-p(H_{2k-1}-H_{k-1}))}{k!(k-1)!}\\
&\equiv\frac{(-1)^k}2\binom{2k}k(1-pH_{2k-1}+pH_{k-1})\pmod {p^2}.
\end{align*}
Then modulo $p^4$ we have
\begin{align*}
\sum_{k=(p+3)/2}^{p-1}G(p,k)\equiv-\frac{15p^2}{4^{4p-4}}\sum_{k=1}^{(p-3)/2}\frac{(1-3pH_{k}+4pH_{2k+1}-2pH_{2k-1}+2pH_{k-1})k\binom{2k}k}{4^k(p-k)(2p-2k-1)},
\end{align*}
where we used $H_{p-1-k}\equiv H_k\pmod p$ for all $k\in\{0,1,\ldots,p-1\}$.\\
It is easy to see that
\begin{align*}
&\sum_{k=1}^{(p-3)/2}\frac{k\binom{2k}k}{4^k(p-k)(2p-2k-1)}\\
&\equiv\sum_{k=1}^{(p-3)/2}\frac{\binom{2k}k}{4^k(2k+1)}+\sum_{k=1}^{(p-3)/2}\frac{2p\binom{2k}k}{4^k(2k+1)^2}+\sum_{k=1}^{(p-3)/2}\frac{p\binom{2k}k}{4^kk(2k+1)}\pmod{p^2}.
\end{align*}
So modulo $p^4$ we have
\begin{align*}
&\sum_{k=\frac{p+3}2}^{p-1}G(p,k)\\
&\equiv-\frac{15p^2}{4^{4p-4}}\left(\sum_{k=1}^{\frac{p-3}2}\frac{\binom{2k}k}{4^k(2k+1)}+\sum_{k=1}^{\frac{p-3}2}\frac{p\binom{2k}k(2H_{2k}-H_k)}{4^k(2k+1)}+\sum_{k=1}^{\frac{p-3}2}\frac{6p\binom{2k}k}{4^k(2k+1)^2}\right).
\end{align*}
Then we immediately get the desired result with Lemma \ref{sun1} and (\ref{H2kHk}).
\end{proof}
\noindent{\it Proof of Theorem \ref{Thsun2}}. Combining (\ref{wz2}) with Lemmas \ref{Fp1p12}, \ref{G13}, \ref{Gpr1} and \ref{Gpr12pr1}, we immediately get that for any prime $p>3$,
$$
\sum_{n=0}^{p-1}\frac{20n+3}{(-2^{10})^n}{4n\choose n,n,n,n}\equiv(-1)^{(p-1)/2}3p+3p^3E_{p-3}\pmod{p^4}.
$$
$p=3$ is easy to check. Therefore the proof of Theorem \ref{Thsun2} is complete.\qed

\vskip 3mm \noindent{\bf Acknowledgments.}
The first author is funded by the Startup Foundation for Introducing Talent of Nanjing University of Information Science and Technology (2019r062). The authors would like to thank Prof. Z.-W. Sun for some helpful comments.


\begin{thebibliography}{99}
\small \setlength{\itemsep}{-.8mm}
\bibitem{g-jmaa-2018}V.J.W. Guo, A $q$-analogue of a Ramanujan-type supercongruence involving central binomial coefficients, J. Math. Anal. Appl. 458 (2018), 590--600.

\bibitem{guo-rama-2019}V.J.W. Guo, $q$-Analogues of the (E.2) and (F.2) supercongruences of van Hamme, Raman. J. 49 (2019), 531--544.

\bibitem{gl-arxiv-2019}V.J.W. Guo and J.-C. Liu, Some congruences related to a congruence of van Hamme, preprint, arxiv:1903.03766.

\bibitem{CXH-rama-2016}Y. G. Chen, X. Y. Xie and B. He, On some congruences of certain binomial sums, Raman. J. 40 (2016), 237--244.

\bibitem{he-jnt-2015}B. He, On the divisibility properties of certain binomial sums, J. Number Theory, 147 (2015), 133--140.

\bibitem{hm-rama-2017}D.-W. Hu and G.-S. Mao, On an extension of a van Hamme supercongruence, Raman. J. 42 (2017), 713--723.

\bibitem{long-2011-pjm} L. Long, Hypergeometric evaluation identities and supercongruences, Pacific J. Math. 249 (2011), no 2, 405--418.

\bibitem{Mor} F. Morley, Note on the congruence $2^{4n}\equiv(-1)^n(2n)!/(n!)^2$, where $2n+1$ is a prime, Ann. Math. 9 (1895), 168--170.

\bibitem{zhsun} Z.-H. Sun, Congruences concerning Bernoulli numbers and Bernoulli polynomials, Discrete. Appl. Math. 105 (2000), 193--223.

\bibitem{sun-jnt-2011}Z.-W. Sun, On congruences related to central binomial coefficients, J. Number Theory, 131 (2011), no. 11, 2219--2238.

\bibitem{sun-scm-2011}Z.-W. Sun, Super congruences and Euler numbers, Sci. China Math. 54 (2011), 2509--2535.

\bibitem{sun-ijm-2012}Z.-W. Sun, A refinement of a congruence result by van Hamme and mortenson, Illinois J. Math. 56 (2012), no. 3, 967--979.

\bibitem{vhamme}L. van Hamme, Some conjectures concerning partial sums of generalized hypergeometric series, in:"$p$-adic functional analysis" (Nijmegen, 1996), 223--236, Lecture Notes in Pure and Appl. Math. 192, Dekker, 1997.

\bibitem{wolstenholme-qjpam-1862}J. Wolstenholme, On certain properties of prime numbers, Quart. J. Pure Appl. Math. 5 (1862), 35--39.

\bibitem{zudilin-jnt-2009}W. Zudilin, Ramanujan-type supercongruences, J. Number Theory, 129 (2009), 1848--1857.
\end{thebibliography}
\end{document}